\documentclass[11pt]{article}
\usepackage{amsmath}
\usepackage{amsfonts}
\usepackage{amssymb}
\usepackage{amsthm}
\usepackage{hyperref}
\usepackage[a4paper, centering, vcentering]{geometry}

\usepackage{amsbsy,amscd,mathrsfs}
\usepackage{url}

\usepackage{color}
\usepackage{fullpage}
\usepackage{pst-grad} 
\usepackage{pst-plot} 

\def\blfootnote{\xdef\@thefnmark{}\@footnotetext}

\long\def\symbolfootnote[#1]#2{\begingroup%
\def\thefootnote{\fnsymbol{footnote}}\footnote[#1]{#2}\endgroup} 

\newtheorem{theorem}{Theorem}
\newtheorem{corollary}[theorem]{Corollary}
\newtheorem{proposition}[theorem]{Proposition}
\newtheorem{lemma}[theorem]{Lemma}
\theoremstyle{definition} 
\newtheorem{definition}[theorem]{Definition}
\newtheorem{example}[theorem]{Example}
\theoremstyle{remark}
\newtheorem{remark}{Remark}
\newcommand{\bt}{\begin{theorem}}
\newcommand{\et}{\end{theorem}}
\newcommand{\bl}{\begin{lemma}}
\newcommand{\el}{\end{lemma}}
\newcommand{\bp}{\begin{proposition}}
\newcommand{\ep}{\end{proposition}}
\newcommand{\bc}{\begin{corollary}}
\newcommand{\ec}{\end{corollary}}
\newcommand{\bdeff}{\begin{definition}}
\newcommand{\edeff}{\end{definition}}
\newcommand{\brem}{\begin{remark}}
\newcommand{\erem}{\end{remark}}
\newcommand{\bex}{\begin{example}}
\newcommand{\eex}{\end{example}}

\newcommand{\lp}{\left(}
\newcommand{\rp}{\right)}

\newcommand{\bi}{\begin{itemize}}
\newcommand{\iii}{\item}
\newcommand{\ei}{\end{itemize}}
\newcommand{\bd}{\begin{description}}
\newcommand{\ed}{\end{description}}

\newcommand{\bqn}{\begin{eqnarray}}
\newcommand{\eqn}{\end{eqnarray}}

\newcommand{\g}{\gamma}

\newcommand{\R}{\mathbb{R}}
\newcommand{\N}{\mathbb{N}}

\newcommand{\mc}[1]{\mathcal{ #1 }}

\newcommand{\all}{\forall\,}

\newcommand{\la}{\langle}
\newcommand{\ra}{\rangle}

\newcommand{\virg}[1]{``#1''}
\newcommand{\tx}[1]{\mathrm{#1}}

\newcommand{\rosso}{}

\newcommand{\ffoot}[1]{}

\newcommand{\VecM}{\tx{Vec}(M)}
\newcommand{\distr}{\mc{D}}
\newcommand{\metr}[2]{\la#1,#2\ra}

\newcommand{\Pg}[1]{\left\{ #1 \right\}}

\newcommand{\lapl}{\Delta}
\newcommand{\dive}{\tx{div}}
\newcommand{\grad}{\nabla}
\newcommand{\popp}{\mc{P}}
\newcommand{\iso}{\tx{Iso}}

\newcommand{\norm}[2]{\|#2\|_{#1}}

\author{Davide Barilari\footnote{CNRS, CMAP, \'Ecole Polytechnique, Paris and \'Equipe INRIA GECO Saclay-\^Ile-de-France, Paris, France - Email: \texttt{barilari@cmap.polytecnique.fr}}  
$\,$ and  Luca Rizzi\footnote{SISSA, Via Bonomea 265, Trieste - Email: \texttt{luca.rizzi@sissa.it}}}
 
\title{A formula for Popp's volume in sub-Riemannian geometry}

\begin{document}
\maketitle

\begin{abstract}
For an equiregular sub-Riemannian manifold $M$, Popp's volume is a smooth volume which is canonically associated with the sub-Riemannian structure, and it is a natural generalization of the Riemannian one. In this paper we prove a general formula for Popp's volume, written in terms of a frame adapted to the sub-Riemannian distribution. As a first application of this result, we prove an explicit formula for the canonical sub-Laplacian, namely the one associated with Popp's volume. Finally, we discuss sub-Riemannian isometries, and we prove that they preserve Popp's volume. We also show that, under some hypotheses on the action of the isometry group of $M$, Popp's volume is essentially the unique volume with such a property. 
\end{abstract}

\section{Introduction}
The problem to define a canonical volume on a sub-Riemannian manifold was first pointed out by Brockett in his seminal paper \cite{brockettctsrg}, motivated by the construction of a Laplace operator on a 3D sub-Riemannian manifold canonically associated with the metric structure, analogous to the Laplace-Beltrami operator on a Riemannian manifold.  
Recently, Montgomery addressed this problem in the general case (see \cite[Chapter 10]{montgomerybook}).

Even on a Riemannian manifold, the Laplacian (defined as the divergence of the gradient) is a second order differential operator whose first order term depends on the choice of the volume on the manifold, which is required to define the divergence. Naively, in the Riemannian case, the choice of a canonical volume is determined by the metric, by requiring that the volume of a orthonormal parallelotope (i.e. whose edges are an orthonormal frame in the tangent space) is 1.

From a geometrical viewpoint, sub-Riemannian geometry is a natural generalization of Riemannian geometry under non-holonomic constraints. Formally speaking, a sub-Riemannian manifold  is a smooth manifold $M$ endowed with a bracket-generating distribution $\distr \subset TM$, with $ k = \text{rank}\,\distr < n = \dim M$, and a smooth fibre-wise scalar product on $\distr$. From this structure, one derives a distance on $M$ - the so-called \emph{Carnot-Caratheodory metric} - as the infimum of the length of \emph{horizontal} curves on $M$, i.e. the curves that are almost everywhere tangent to the distribution. 

Nevertheless, sub-Riemannian geometry enjoys major differences with respect to the Riemannian case. For instance, 
a construction analogue to the one described above for the Riemannian volume is not possible. 
Indeed the inner product is defined only on a subspace of the tangent space, and there is no canonical way to extend it on the whole tangent space.
 
Popp's volume is a generalization of the Riemannian volume in sub-Riemannian setting. It was first defined by Octavian Popp but introduced only in \cite{montgomerybook} (see also \cite{hypoelliptic}). Such a volume is smooth only for an equiregular sub-Riemannian manifold, i.e. when the dimensions of the higher order distributions $\distr^{1}:=\distr$, $\distr^{i+1}:=\distr^{i}+[\distr^{i},\distr]$, for every $i\geq1$, do not depend on the point (for precise definitions, see Sec. 2).

Under the equiregularity hypothesis, the bracket-generating condition guarantees that there exists a minimal $m\in \N$, called \emph{step} of the structure, such that $\distr^{m}=TM$.

Then, for each $q \in M$, it is well defined the graded vector space: 
\bqn\label{eq:grq}
\text{gr}_{q}(\distr):=\bigoplus_{i=1}^{m} \distr^{i}_{q}/\distr^{i-1}_{q}, \qquad\tx{where} \ \  \distr_{q}^{0}=0.
\eqn
The vector space $\text{gr}_q(\distr)$, which can be endowed with a natural sub-Riemannian structure, is called the \emph{nilpotentization} of the structure at the point $q$, and plays a role analogous to the Euclidean tangent space in Riemannian geometry.
Popp's volume is defined  by inducing a canonical inner product on $\text{gr}_{q}(\distr)$ via the Lie brackets, and then using a \emph{non-canonical} isomorphism between $\text{gr}_{q}(\distr)$ and $T_{q}M$ to define an inner product on the whole $T_q M$. Interestingly, even though this construction depends on the choice of some complement to the distribution, the associated volume form (i.e. Popp's volume) is independent on this choice.

It is worth to recall that on a sub-Riemannian manifold, which is a metric space, the Haussdorff volume and the spherical Hausdorff volume, respectively  $\mc{H}^{Q}$  and $\mc{S}^{Q}$, are canonically defined.\footnote{Recall that the Hausdorff dimension of a sub-Riemannian manifold $M$ is given by the formula 
$Q=\sum_{i=1}^m i n_i$, where $\ n_i:=\tx{dim}\, \distr_{q}^i/\distr_{q}^{i-1}$.
In particular the Hausdorff dimension is always bigger than the topological dimension.}
The relation between Popp's volume and $\mc{S}^{Q}$ has been studied in \cite{corank1}, where the authors show how the Radon-Nikodym derivative is related with the nilpotentization of the structure. In particular they prove that  the Radon-Nikodym derivative could also be non smooth (see also \cite{corank2,corank2g}). Remember that the Hausdorff and spherical Hausdorff volumes are both proportional to the Riemannian one on a Riemannian manifold. The relation between Hausdorff measures for curves and different notions of length in sub-Riemannian geometry is also investigated in \cite{ghezzijean}.

On a contact sub-Riemannian manifold, Popp's volume coincides with the Riemannian volume obtained by \virg{promoting} the Reeb vector field to an orthonormal complement to the distribution. In the general case, unfortunately, the definition is more involved. To the authors' best knowledge, explicit formul{\ae}  for Popp's volume appeared, for some specific cases, only in \cite{corank1,corank2,corank2g}. 

The goal of this paper is to prove a general formula for Popp's volume, in terms of any \emph{adapted} frame of the tangent bundle. In order to present the main results here, we briefly introduce some concepts which we will elaborate in details in the subsequent sections. Thus, we say that a local frame $X_{1},\ldots,X_{n}$ is adapted if $X_{1},\ldots,X_{k_{i}}$ is a local frame for $\distr^{i}$, where  $k_{i}:=\dim \distr^{i}$, and $X_1,\dots,X_k$ are orthonormal. Even though it is not needed right now, it is useful to define the functions $c^l_{ij} \in C^{\infty}(M)$ by
\bqn\label{eq:sc0intro}
[X_i,X_j] = \sum_{l=1}^n c_{ij}^l X_l\,.
\eqn
With a standard abuse of notation we call them \emph{structure constants}. For $j=2,\ldots,m$ we define the \emph{adapted structure constants} $b^l_{i_1\dots \,i_j}\in C^{\infty}(M)$ as follows:
\bqn \label{eq:sc1intro}
[X_{i_1},[X_{i_2}, \dots,[X_{i_{j-1}},X_{i_j}] ]] = \sum_{l = k_{j-1}+1}^{k_j} b^l_{i_1 i_2 \dots \,i_j} X_l \mod   \distr^{j-1}\,,
\eqn
where $1\leq i_1,\dots,i_j\leq k$. These are a generalization of the $c_{ij}^l$, with an important difference: the structure constants of Eq.~\eqref{eq:sc0intro} are obtained by considering the Lie bracket of all the fields of the local frame, namely $1\leq i,j,l \leq n$. On the other hand, the adapted structure constants of Eq.~\eqref{eq:sc1intro} are obtained by taking the iterated Lie brackets of the first $k$ elements of the adapted frame only (i.e. the local orthonormal frame for $\distr$), and considering the appropriate equivalence class. For $j=2$, the adapted structure constants can be directly compared to the standard ones. Namely $b^l_{ij} = c_{ij}^l$ when both are defined, that is for $1\leq i,j \leq k$, $l\geq k+1$.

Then, we define the $k_j-k_{j-1}$ dimensional square matrix $B_{j}$ as follows:
\begin{equation}\label{eq:Bintro}
\left[B_{j}\right]^{hl} = \sum_{i_1,i_2,\dots,i_j = 1}^{k} b^h_{i_1 i_2 \dots i_j} b^l_{i_1 i_2 \dots i_j}\,,\qquad j =1,\dots,m\,,
\end{equation}
with the understanding that $B_{1}$ is the $k\times k$ identity matrix. It turns out that each $B_j$ is positive definite.

\bt \label{t:poppmain} Let $X_{1},\ldots,X_{n}$ be a local adapted frame, and let $\nu ^1,\dots,\nu^n$ be the dual frame. Then Popp's volume $\popp$ satisfies
\bqn\label{eq:poppmainintro}
\popp=\frac{1}{\sqrt{\prod_{j} \det B_{j}}} \,\nu ^1 \wedge \ldots \wedge \nu^n \,,
\eqn
where $B_{j}$ is defined by \eqref{eq:Bintro} in terms of the adapted structure constants \eqref{eq:sc1intro}.
\et
To clarify the geometric meaning of Eq.~\eqref{eq:poppmainintro}, let us consider more closely the case $m = 2$. If $\distr$ is a step $2$ distribution, we can build a local adapted frame $\{X_1,\dots,X_k,X_{k+1},\dots,X_n\}$ by completing any local orthonormal frame $\{X_1,\dots,X_k\}$ of the distribution to a local frame of the whole tangent bundle. Even though it may not be evident, it turns out that $B_2^{-1}(q)$ is the Gram matrix of the vectors $X_{k+1},\dots,X_n$, seen as elements of $T_qM/ \distr_q$. The latter has a natural structure of inner product space, induced by the surjective linear map $[\,,\,]:\distr_q\otimes\distr_q \to T_qM/\distr_q$ (see Lemma~\ref{l:mont1}). Therefore, the function appearing at the beginning of Eq.~\eqref{eq:poppmainintro} is the volume of the parallelotope whose edges are $X_1,\dots,X_n$, seen as elements of the orthogonal direct sum $\text{gr}_q(\distr)= \distr_q \oplus T_qM/\distr_q$.

With a volume form at disposal, one can naturally define the associated divergence operator, which acts on vector fields. Moreover, the sub-Riemannian structure allows to define the horizontal gradient of a smooth function. Then, we define a canonical sub-Laplace operator as $\lapl := \dive \circ \grad$, which generalizes the Laplace-Beltrami operator. This is a second order differential operator, which has been studied in \cite{hypoelliptic,mioheat}.
As a corollary to Theorem~\ref{t:poppmain}, we obtain a formula for the sub-Laplacian $\lapl$ in terms of any local adapted frame.

\bc \label{c:sublapl} Let $X_{1},\ldots,X_{n}$ be a local adapted frame. Let $\lapl$ be the canonical sub-Laplacian. Then
\bqn \label{eq:sublapl}
\lapl = \sum_{i=1}^k  X_i^2 - \lp \frac{1}{2}\sum_{j=1}^m \tx{Tr}(B^{-1}_j X_i(B_j)) + \sum_{l=1}^n c_{il}^l\rp  X_i\,,
\eqn
where $c_{ij}^l$ are the structure constants \eqref{eq:sc0intro}, and $B_{j}$ is defined by \eqref{eq:Bintro} in terms of the adapted structure constants \eqref{eq:sc1intro}.
\ec
\brem
If $M$ is a Carnot group (i.e. a connected, simply connected nilpotent group, whose Lie algebra is graded, and whose sub-Riemannian structure is left invariant) the $B_j$ are constant. Moreover, $\forall i$ $\sum_{l=1}^n c_{il}^l = 0$, as a consequence of the graded structure. Then, in this case, the sub-Laplacian is a simple \virg{sum of squares} $\lapl = \sum_{i=1}^k X_i^2$. 
This is a manifestation of the fact that Carnot groups are to sub-Riemannian geometry as Euclidean spaces are to Riemannian geometry. Indeed, on $\mathbb{R}^n$, the Laplace-Beltrami operator is a simple sum of squares.

More in general, in \cite{hypoelliptic}, the authors prove that for left-invariant structures on unimodular Lie groups the sub-Laplacian is a sum of squares.
\erem

In the last part of the paper we discuss the conditions under which a local isometry preserves Popp's volume. In the Riemannian setting, an isometry is a diffeomorphism such that its differential is an isometry for the Riemannian metric. The concept is easily generalized to the sub-Riemannian case.
\begin{definition}
A (local) diffeomorphism $\phi: M \to M$ is a \emph{(local) isometry} if its differential $\phi_* : TM \to TM$ preserves the sub-Riemannian structure $(\distr,\metr{\cdot}{\cdot})$, namely
\begin{itemize}
\item[i)] $\phi_*(\distr_q) = \distr_{\phi(q)}$ for all $q \in M$,
\item[ii)] $\metr{\phi_* X}{\phi_* Y}_{\phi(q)} = \metr{X}{Y}_q$ for all $q\in M$, $X,Y \in \distr_q$\,.
\end{itemize}
\end{definition}
\begin{remark} Condition $i)$, which is trivial in the Riemannian case, is necessary to define isometries in the sub-Riemannian case. Actually, it also implies that all the higher order distributions are preserved by $\phi_*$, i.e. $\phi_*(\distr^i_q) = \distr^i_{\phi(q)}$, for $1\leq i \leq m$.
\end{remark}
\begin{definition}
Let $M$ be a manifold equipped with a volume form $\mu \in \Omega^n(M)$. We say that a (local) diffeomorphism $\phi:M \to M$ is a \emph{(local) volume preserving transformation} if $\phi^* \mu = \mu$.
\end{definition}
In the Riemannian case, local isometries are also volume preserving transformations for the Riemannian volume. Then, it is natural to ask whether this is true also in the sub-Riemannian setting, for some choice of the volume. The next \rosso{proposition} states that the answer is positive if we choose Popp's volume.

\begin{proposition}\label{p:volumepres}
Sub-Riemannian (local) isometries are volume preserving transformations for Popp's volume.
\end{proposition}
\rosso{Proposition}~\ref{p:volumepres} may be false for volumes different than Popp's one. We have the following.
\begin{proposition}\label{p:volumetrans}
Let $\iso (M)$ be the group of isometries of the sub-Riemannian manifold $M$. If $\iso (M)$ acts transitively on $M$, then Popp's volume is the unique volume (up to multiplication by scalar constant) such that \rosso{Proposition}~\ref{p:volumepres} holds true.
\end{proposition}

\begin{definition}
Let $M$ be a Lie group. A sub-Riemannian structure $(M,\distr,\metr{\cdot}{\cdot})$ is left invariant if $\forall g \in M$, the left action $L_g:M\to M$ is an isometry.
\end{definition}

As a trivial consequence of \rosso{Proposition}~\ref{p:volumepres} we recover a well-known result (see again \cite{montgomerybook}).
\begin{corollary}\label{c:haar}
Let $(M,\distr, \metr{\cdot}{\cdot})$ be a left-invariant sub-Riemannian structure. Then Popp's volume is left invariant, i.e. $L_g^* \popp = \popp$ for every $g \in M$.
\end{corollary}

\rosso{Propositions} \ref{p:volumepres}, \ref{p:volumetrans} and Corollary~\ref{c:haar} should shed some light about which is the \virg{most natural} volume for sub-Riemannian manifold.

\section{Sub-Riemannian geometry}

We recall some basic definitions in sub-Riemannian geometry. For a more detailed introduction, see \cite{nostrolibro,bellaiche,montgomerybook}.
\bdeff
A \emph{sub-Riemannian manifold} is a triple $(M,\distr,\metr{\cdot}{\cdot})$, 
where
\bi
\iii[$(i)$] $M$ is a connected orientable smooth manifold of dimension $n\geq 3$;
\iii[$(ii)$] $\distr \subset TM$ is a smooth distribution of constant rank $k< n$;
\iii[$(iii)$] $\metr{\cdot}{\cdot}_q$ is an inner product on the fibres $\distr_{q}$, smooth as a function of $q$. 
\ei
\edeff
Let $\Gamma(\distr)\subset \VecM$ be the $C^{\infty}(M)$-module of the smooth sections of $\distr$. Throughout this paper we assume that the sub-Riemannian manifold $M$ satisfies the \emph{bracket-generating condition}, i.e.
\bqn \label{Hor}
\text{span}\{[X_1,[X_2,\ldots,[X_{j-1},X_j]]](q)~|~X_i\in\Gamma(\distr),\, j\in \N\}=T_qM, \quad \all q\in M.
\eqn
In other words, the iterated Lie brackets of smooth sections of $\distr$ span the whole tangent bundle $TM$. Condition \eqref{Hor} is also called \emph{H\"ormander condition}, and bracket-generating distribution are also referred to as \emph{completely nonholonomic} distributions.

\rosso{An absolutely continuous} curve $\g:[0,T]\to M$ is said to be \emph{horizontal} (or \emph{admissible}) if 
$$\dot\g(t)\in\distr_{\g(t)}\qquad \text{ for a.e. } t\in[0,T].$$

Given an horizontal curve $\g:[0,T]\to M$, the {\it length of $\g$} is
\bqn
\label{e-lunghezza}
\ell(\g)=\int_0^T |\dot{\g}(t)|~dt.
\eqn
The {\it distance} induced by the sub-Riemannian structure on $M$ is the 
function
\bqn
\label{e-dipoi}
d(q_0,q_1)=\inf \{\ell(\g)\mid \g(0)=q_0,\g(T)=q_1, \g\ \mathrm{horizontal}\}.
\eqn
The connectedness of $M$ and the bracket-generating condition guarantee the finiteness and the continuity of the sub-Riemannian distance with respect to the topology of $M$ (Chow-Rashevsky Theorem, see, for instance, \cite{agrachevbook}). The function $d(\cdot,\cdot)$ is called the \emph{Carnot-Caratheodory distance} and gives to $M$ the structure of metric space (see \cite{bellaiche,gromov}).

Locally (i.e. on an open set $U\subset M$), there always exists a set of $k$ smooth vector fields $X_1,\dots,X_k$ such that, $\forall q \in U$, it is an orthonormal basis of $\distr_q$. The set $\Pg{X_1,\ldots,X_k}$ is called a \emph{local orthonormal frame} for the sub-Riemannian structure.

\bdeff
Let $\distr$ be a distribution. Its \emph{flag at $q\in M$} is the sequence of vector spaces $\distr_q^0\subset\distr_q^{1}\subset\distr_q^{2}\subset\ldots\subset T_qM$ defined by
\begin{equation*}
\distr^0_q:= \{0\},\qquad\distr_q^{1}:=\distr_q,\qquad\distr_q^{i+1}:=\distr_q^{i}+[\distr^{i},\distr]_q\,,
\end{equation*}
where, with a standard abuse of notation, we understand that $[\distr^i,\distr]_q$ is the vector space generated by the iterated Lie brackets, up to length $i$, of local sections of the distribution, evaluated at $q$.
\edeff
Even though the rank of $\distr$ is constant, the dimensions of the subspaces of the flag, i.e. the numbers $k_i(q):= \dim(\distr_q^i)$ may depend on the point. Observe that the bracket-generating condition can be rewritten as follows
\begin{equation*}
\forall q \in M \ \ \exists  \ \  \text{minimal} \ \  m(q)\in \N \quad \text{ such that }\quad k_m(q) = \dim\,T_qM.
\end{equation*}
The number $m(q)$ is called the \emph{step} of the distribution at the point $q$. The vector $\mc{G}(q):=(k_1(q),k_2(q),\dots,k_m(q))$ is called the \emph{growth vector} of the distribution at $q$.

\bdeff
A distribution $\distr$ is \emph{equiregular} if the growth vector is constant, i.e. for each $i=1,2,\ldots,m$, $k_i(q) = \dim (\distr^{i}_q)$ does not depend on $q\in M$. In this case the subspaces $\distr^i_q$ are fibres of the \emph{higher order distributions} $\distr^i \subset TM$.
\edeff
For equiregular distributions we will simply talk about growth vector and step of the distribution, without any reference to the point $q$. 

Finally, we introduce the nilpotentization of the distribution at the point $q$, which is fundamental for the definition of Popp's volume.
\bdeff Let $\distr$ be an equiregular distribution of step $m$. The \emph{nilpotentization of $\distr$} at the point $q\in M$ is the graded vector space
\begin{equation*}
\text{gr}_q(\distr) = \distr_q \oplus \distr_q^2/\distr_q \oplus \ldots \oplus \distr_q^m/\distr_q^{m-1}.
\end{equation*}
\edeff

The vector space $\text{gr}_q(\distr)$ can be endowed with a Lie algebra structure, which respects the grading. Then, there is a unique connected, simply connected group, $\text{Gr}_q(\distr)$, such that its Lie algebra is $\text{gr}_q(\distr)$. The global, left-invariant vector fields obtained by the group action on any orthonormal basis of $\distr_q\subset \text{gr}_q(\distr)$ defines a sub-Riemannian structure on $\text{Gr}_q(\distr)$, which is called the \emph{nilpotent approximation} of the sub-Riemannian structure at the point $q$.

\section{Popp's volume}\label{s:popp}

In this section we provide the definition of Popp's volume, and we prove Theorem~\ref{t:poppmain}. Our presentation follows closely the one \rosso{that} can be found in \cite{montgomerybook}. The definition rests on the following lemmas.

\bl \label{l:mont1} 
Let $E$ be an inner product space, and let $\pi:E\to V$ be a surjective linear map. Then $\pi$ induces an inner product on $V$ \rosso{such that the length of $v \in V$ is}
\bqn \label{eq:final}
\|v\|_V = \min\{ \|e\|_E \text{ s.t. } \pi(e) = v \}\,.
\eqn
\el
\begin{proof}
\rosso{It is easy to check that Eq. \eqref{eq:final} defines a norm on $V$. Moreover, since $\|\cdot\|_{E}$ is induced by an inner product, i.e. it satisfies the parallelogram identity, it follows that $\|\cdot\|_{V}$ satisfies the parallelogram identity too. Notice that this is equivalent to consider the inner product on $V$ defined by the linear isomorphism $\pi: (\ker \pi)^\perp \to V$. Indeed the length of $v \in V$ is the length of the shortest element $e \in \pi^{-1}(v)$.}
\end{proof}
\bl \label{l:mont2} 
Let $E$ be a vector space of dimension $n$ with a flag of linear subspaces $\{0\} = F^0 \subset F^1\subset F^2 \subset \ldots\subset F^m = E$. Let $\tx{gr}(F) = F^1\oplus F^2/F^1\oplus \ldots \oplus F^m/F^{m-1}$ be the associated graded vector space. Then there is a canonical isomorphism $\theta: \wedge^n E \to \wedge^n \tx{gr}(F)$. 
\el
\begin{proof}
We only give a sketch of the proof. For $0\leq i \leq m$, let $k_i:= \dim F^i$. Let $X_1,\dots,X_n$ be a adapted basis for $E$, i.e. \rosso{$X_1,\dots, X_{k_i}$} is a basis for $F^i$. We define the linear map $\widehat{\theta}: E \to \tx{gr}(F)$ which, for $0\leq j \leq m-1$, takes $X_{k_j+1}, \dots, X_{k_{j+1}}$ to the corresponding equivalence class in $F^{j+1}/F^j$. This map is indeed a non-canonical isomorphism, which depends on the choice of the adapted basis. In turn, $\widehat{\theta}$ induces a map $\theta : \wedge^n E \to \wedge^n \tx{gr}(F)$, which sends $X_1\wedge\ldots\wedge X_n$ to $\widehat{\theta}(X_1)\wedge\ldots\wedge\widehat{\theta}(X_n)$. The proof that $\theta$ does not depend on the choice of the adapted basis is \virg{dual} to \cite[Lemma 10.4]{montgomerybook}.
\end{proof}

The idea behind Popp's volume is to define an inner product on each $\distr^i_q/\distr^{i-1}_q$ which, in turn, induces an inner product on the orthogonal direct sum $\text{gr}_q(\distr)$. The latter has a natural volume form, which is the canonical volume of an inner product space obtained by wedging the elements an orthonormal dual basis. Then, we employ Lemma~\ref{l:mont2} to define an element of $(\wedge^n T_q M)^*\simeq \wedge^n T_q^* M$, which is Popp's volume form computed at $q$.

Fix $q \in M$. Then, let $v,w \in \distr_q$, and let $V,W$ be any horizontal extensions of $v,w$. Namely, $V,W \in \Gamma(\distr)$ and $V(q) =v$, $W(q) = w$. The linear map $\pi : \distr_q\otimes\distr_q \to \distr_q^2/\distr_q$
\begin{equation}
\pi(v\otimes w):= [V,W]_q \mod \distr_q\,,
\end{equation} 
is well defined, and does not depend on the choice the horizontal extensions. Indeed let $\widetilde{V}$ and $\widetilde{W}$ be two different horizontal extensions of $v$ and $w$ respectively. Then, in terms of a local frame $X_1,\dots,X_k$ of $\distr$
\begin{equation}
\widetilde{V} = V + \sum_{i=1}^k f_i X_i\,,\qquad 
\widetilde{W} = W + \sum_{i=1}^k g_i X_i\,,
\end{equation}
where, for $1\leq i \leq k$, $f_i,g_i \in C^\infty(M)$ and $f_i(q) = g_i(q) = 0$. Therefore
\bqn
[\widetilde{V}, \widetilde{W}] = [V,W] + \sum_{i=1}^k  \left(V(g_i) - W(f_i)\right )   X_i +  \sum_{i,j=1}^k f_i g_j [X_i,X_j] \,.
\eqn
Thus, evaluating at $q$, $[\widetilde{V},\widetilde{W}]_q = [V,W]_q \mod \distr_q$, as claimed.
Similarly, let $1\leq i \leq m$. The linear maps $\pi_i: \otimes^i \distr_q \to \distr_q^i/\distr_q^{i-1}$
\begin{equation}\label{eq:pimap}
\pi_i(v_1\otimes\dots\otimes v_i) = [V_1,[V_2,\dots,[V_{i-1},V_i]]]_q \mod \distr^{i-1}_q\,,
\end{equation}
are well defined and do not depend on the choice of the horizontal extensions $V_1,\dots,V_i$ of $v_1,\dots,v_i$.

By the bracket-generating condition, $\pi_i$ are surjective and, by Lemma~\ref{l:mont1}, they induce an inner product space structure on $\distr_q^i/\distr_q^{i-1}$. Therefore, the nilpotentization of the distribution at $q$, namely 
\bqn
\text{gr}_q(\distr) = \distr_q \oplus \distr_q^2/\distr_q \oplus \ldots \oplus \distr_q^m/\distr_q^{m-1}\,,
\eqn
is an inner product space, as the orthogonal direct sum of a finite number of inner product spaces. As such, it is endowed with a canonical volume (defined up to a sign) $\mu_q \in \wedge^n\text{gr}_q(\distr)^*$, which is the volume form obtained by wedging the elements of an orthonormal dual basis.

Finally, Popp's volume (computed at the point $q$) is obtained by transporting the volume of $\text{gr}_q(\distr)$ to $T_q M$ through the map $\theta_q :\wedge^n T_q M \to \wedge^n \text{gr}_q(\distr)$ defined in Lemma~\ref{l:mont2}. Namely 
\bqn\label{eq:popppoint}
\popp_q =\theta_{q}^{*}(\mu_{q})= \mu_q \circ \theta_q\,,
\eqn
where $\theta_{q}^{*}$ denotes the dual map and  we employ the canonical identification $(\wedge^n T_q M)^* \simeq \wedge^n T^*_q M$. Eq.~\eqref{eq:popppoint} is defined only in the domain of the chosen local frame. Since $M$ is orientable, with a standard argument, these $n$-forms can be glued together to obtain Popp's volume $\popp \in \Omega^n(M)$. The smoothness of $\popp$ follows directly from Theorem~\ref{t:poppmain}.

\brem The definition of Popp's volume can be restated as follows.
Let $(M,\distr)$ be an oriented sub-Riemannian manifold. Popp's volume is the unique volume $\popp$ such that, for all $q \in M$, the following diagram is commutative:
\begin{displaymath}
\begin{CD}
(M,\distr) @>{\popp}>> (\wedge^n T_q M)^* \\
@V{\tx{gr}_q}VV @VV{\theta^*_q}V\\
\tx{gr}_q(\distr) @>>\mu > (\wedge^n \tx{gr}_q(\distr))^*
\end{CD}
\end{displaymath}
where $\mu$ associates the inner product space $\tx{gr}_q(\distr)$ with its canonical volume $\mu_q$, and $\theta^*_q$ is the dual of the map defined in Lemma~\ref{l:mont2}.

\erem

\subsection{Proof of Theorem~\ref{t:poppmain}}
We are now ready to prove Theorem~\ref{t:poppmain}. For convenience, we first prove it for a distribution of step $m =2$. Then, we discuss the general case. In the following subsections, everything is understood to be computed at a fixed point $q\in M$. Namely, by $\text{gr}(\distr)$ we mean the nilpotentization of $\distr$ at the point $q$, and by $\distr^i$ we mean the fibre $\distr^i_q$ of the appropriate higher order distribution.

\subsubsection{Step $2$ distribution}

If $\distr$ is a step $2$ distribution, then $\distr^2 = T M$. The growth vector is $\mc{G}=(k,n)$. We choose $n-k$ independent vector fields $\{Y_l\}_{l=k+1}^n$  such that $X_1,\dots,X_k,Y_{k+1},\dots,Y_n$ is a local adapted frame for $TM$. Then
\begin{equation}\label{struc1}
[X_i,X_j] = \sum_{l=k+1}^n b_{ij}^l Y_l \mod \distr \,.
\end{equation}
For each $l = k+1,\dots,n$, we can think to $b^l_{ij}$ as the components of an Euclidean vector in $\mathbb{R}^{k^2}$, which we denote by the symbol $b^l$. According to the general construction of Popp's volume, we need first to compute the inner product on the orthogonal direct sum $\text{gr}(\distr) = \distr \oplus \distr^2 / \distr$. By Lemma~\ref{l:mont1}, the norm on $\distr^2/\distr$ is induced by the linear map $\pi: \otimes^2 \distr  \to \distr^2/\distr$ 
\begin{equation}
\pi(X_i\otimes X_j )= [X_i,X_j] \mod \distr \,.
\end{equation}
The vector space $\otimes^2 \distr$ inherits an inner product from the one on $\distr$, namely $\forall X,Y,Z,W \in \distr$, $\langle X\otimes Y, Z\otimes W\rangle = \langle X,Z\rangle\langle Y,W\rangle$. $\pi$ is surjective, then we identify the range $\distr^2/\distr$ with $\ker\pi^\perp \subset \otimes^2\distr$, and define an inner product on $\distr^2/\distr$ by this identification.
In order to compute explicitly the norm on $\distr^2/\distr$ (and then, by polarization, the inner product), let $Y\in\distr^2/\distr$. 
Then
\begin{equation}\label{eq:normstep2}
\norm{\distr^2/\distr}{Y} = \min\{\norm{\otimes^2\distr}{Z} \text{ s.t. } \pi(Z) = Y \}\,.
\end{equation}
Let $Y = \sum_{l=k+1}^n c^l Y_l$ and $Z = \sum_{i,j=1}^k a_{ij} X_i\otimes X_j \in \otimes^2\distr$. We can think to $a_{ij}$ as the components of a vector $a \in \mathbb{R}^{k^2}$. Then, Eq.~\eqref{eq:normstep2} writes
\begin{equation}
\norm{\distr^2/\distr}{Y} = \min\{|a| \text{ s.t. } a\cdot b^l = c^l,\, l=k+1,\dots,n \}\,,
\end{equation}
where $|a|$ is the Euclidean norm of $a$, and the dot denotes the Euclidean inner product. Indeed, $\norm{\distr^2/\distr}{Y}$ is the Euclidean distance of the origin from the affine subspace of $\mathbb{R}^{k^2}$ defined by the equations $a\cdot b^l = c^l$ for $l=k+1,\dots,n$. In order to find an explicit expression for $\norm{\distr^2/\distr}{Y}^2$ in terms of the $b^l$, we employ the Lagrange multipliers technique. Then, we look for extremals of 
\bqn
L(a,b^{k+1},\dots,b^n,\lambda_{k+1},\dots,\lambda_n) = |a|^2 -2 \sum_{l=k+1}^n \lambda_l (a\cdot b^l - c^l)\,.
\eqn
We obtain the following system
\begin{equation}\label{lag}
\begin{cases}
\displaystyle\sum_{l=k+1}^n \lambda_l \cdot b^l -a = 0, & \\
\displaystyle \sum_{l=k+1}^n \lambda_l b^l\cdot b^r = c^r\,, \qquad r = k+1,\dots,n. & 
\end{cases}
\end{equation}
Let us define the $n-k$ square matrix $B$, with components $B^{hl} = b^h\cdot b^l$. $B$ is a Gram matrix, which is positive definite iff the $b^l$ are $n-k$ linearly independent vectors. These vectors are exactly the rows of the representative matrix of the linear map $\pi: \otimes^2\distr \to \distr^2/\distr$, which has rank $n-k$. Therefore $B$ is symmetric and positive definite, hence invertible. It is now easy to write the solution of system \eqref{lag} by employing the matrix $B^{-1}$, which has components $B^{-1}_{hl}$. Indeed a straightforward computation leads to
\begin{equation}
\norm{\distr^2/\distr}{c^s Y_s}^2 = c^h  B^{-1}_{hl} c^l \, .
\end{equation}
By polarization, the inner product on $\distr^2/\distr$ is defined, in the basis $Y_l$, by
\begin{equation}
\langle Y_l, Y_h\rangle_{\distr^2/\distr} = B^{-1}_{lh}\,.
\end{equation}
Observe that $B^{-1}$ is the Gram matrix of the vectors $Y_{k+1},\dots,Y_n$ seen as elements of $\distr^2/\distr$. Then, by the definition of Popp's volume, if $\nu^1,\dots,\nu^k,\mu^{k+1},\dots,\mu^n$ is the dual basis associated with $X_1,\dots,X_k,Y_{k+1},\dots,Y_n$, the following formula holds true
\begin{equation}
\popp = \frac{1}{\sqrt{\det{B}}}\, \nu^1 \wedge \dots \wedge \nu^k \wedge \mu^{k+1} \wedge \dots \wedge \mu^n\,.
\end{equation}

\subsubsection{General case}

In the general case, the procedure above can be carried out with no difficulty. Let $X_1,\dots,X_n$ be a local adapted frame for the flag $\distr^0 \subset \distr \subset\distr^2\subset\dots\subset \distr^m$. As usual $k_i = \dim(\distr^i)$. For $j=2,\ldots,m$ we define the adapted structure constants $b^l_{i_1\dots \,i_j}\in C^{\infty}(M)$ by
\bqn \label{eq:sc1}
[X_{i_1},[X_{i_2}, \dots,[X_{i_{j-1}},X_{i_j}]]] = \sum_{l = k_{j-1}+1}^{k_j} b^l_{i_1 i_2 \dots \,i_j} X_l \mod   \distr^{j-1}\,,
\eqn
where $1\leq i_1,\dots,i_j\leq k$. Again, $b^l_{i_1\dots i_j}$ can be seen as the components of a vector $b^l \in \mathbb{R}^{k^j}$.

Recall that for each $j$ we defined the surjective linear map $\pi_j:\otimes^j \distr \to \distr^j/\distr^{j-1}$
\begin{equation}
\pi_j(X_{i_1}\otimes X_{i_2}\otimes \dots \otimes X_{i_j}) = [X_{i_1},[X_{i_2}, \dots,[X_{i_{j-1}},X_{i_j}] ]] \mod \distr^{j-1}\,.
\end{equation}
Then, we compute the norm of an element of $\distr^j/\distr^{j-1}$ exactly as in the previous case. It is convenient to define, for each $1\leq j \leq m$, the $k_j-k_{j-1}$ dimensional square matrix $B_{j}$, of components
\begin{equation}\label{Binverse}
\left[B_{j}\right]^{hl} = \sum_{i_1,i_2,\dots,i_j = 1}^{k} b^h_{i_1 i_2 \dots i_j} b^l_{i_1 i_2 \dots i_j}\,.
\end{equation}
with the understanding that $B_1$ is the $k\times k$ identity matrix. Each one of these matrices is symmetric and positive definite, hence invertible, due to the surjectivity of $\pi_j$. The same computation of the previous case, applied to each $\distr^j/\distr^{j-1}$ shows that the matrices $B_j^{-1}$ are precisely the Gram matrices of the vectors $X_{k_{j-1}+1},\dots,X_{k_j} \in \distr^j/\distr^{j-1}$, in other words
\begin{equation}
\langle X_{k_{j-1}+l},X_{k_{j-1}+h}\rangle_{\distr^j/\distr^{j-1}} = B^{-1}_{lh}\,.
\end{equation}
Therefore, if $\nu^1,\dots,\nu^n$ is the dual frame associated with $X_1,\dots,X_n$, Popp's volume is
\begin{equation}
\popp = \frac{1}{\sqrt{\prod_{j=1}^m\det B_j}}\; \nu^1\wedge \ldots \wedge \nu^n \,.
\end{equation}
\subsection{Examples}
In this section we compute Popp's volume for some specific equiregular sub-Riemannian structures. We also discuss, through an example, the non-equiregular case.

\subsubsection{Contact manifolds}

Contact manifolds are a well-known class of sub-Riemannian structures. We recall the basic definition first, then we compute Popp's volume in terms of a canonical operator associated with a contact structure.

\bdeff
Let $\omega \in \Omega^1(M)$ be a one-form on $M$. Let $\distr$ be the $n-1$ dimensional distribution $\distr := \ker\omega$. We say that $\omega$ is a \emph{contact form} if $d\omega|_\distr$ is non degenerate. In this case, $\distr$ is a called \emph{contact distribution}.  A sub-Riemannian structure $(M,\distr,\metr{\cdot}{\cdot})$, where $\distr$ is a contact distribution, is is called \emph{contact sub-Riemannian manifold}.
\edeff
Notice that the non-degeneracy assumption implies that the dimension of $M$ is odd.
Observe that any contact manifold satisfies the bracket-generating condition, is equiregular, has step $2$, and its growth vector is $\mc{G} = (n-1,n)$.

With any contact form $\omega$ we can associate its \emph{Reeb vector field}, which is the unique vector field $X_{0}$ satisfying the conditions $\omega(X_{0})=1$ and $d\omega(X_{0},\cdot)=0$. Notice that, given a local orthonormal frame $X_{1},\ldots,X_{k}$ for the distribution, then $X_{1},\ldots,X_{k},X_{0}$ is a local  adapted frame, since $X_{0}$ is transversal to $\distr$.

The contact form $\omega$ induces a linear bundle map (i.e. a fibre-wise linear map) $J : \distr \to \distr$, defined by $\metr{J X}{Y} = d\omega(X,Y)$, $\forall X,Y \in \distr$.
Observe that the restriction $J_q$ of $J$ to the fibres of $\distr$ is a linear skew-symmetric operator on the inner product space $(\distr_q,\metr{\cdot}{\cdot}_q)$. Hence its Hilbert-Schmidt norm $\|J_{q}\|$ is well defined by the formula $\|J_{q}\|^{2} = \sum_{i,j=1}^k \metr{X_i}{J X_j}^2 $.

\bp Let $M$ be a contact sub-Riemannian manifold and $J:\distr \to \distr$ as above. Let $\nu^{1},\ldots,\nu^{k},\nu^{0}$ be the dual frame associated with the local adapted frame $X_{1},\ldots,X_{k},X_{0}$,  where $X_{0}$ is the Reeb vector field. Then
\bqn\label{eq:poppcontact}
\popp=\frac{1}{\|J_{q}\|}\nu^{1}\wedge\ldots\wedge\nu^{k}\wedge\nu^{0}\,,
\eqn
where $\|J_{q}\|$ is the Hilbert-Schmidt norm of $J_{q}$. \ep
\begin{proof}
Let $X_{1},\ldots,X_{k},X_{0}$ be a local adapted frame, where $X_{0}$ is the Reeb vector field associated with the contact form. Then, for $1\leq i,j\leq k$, the structure constants satisfy
\begin{align}
&[X_{i},X_{j}]=\sum_{l=1}^{k}c_{ij}^{l}X_{l}+c_{ij}^{0}X_{0}\,,\\
&[X_{i},X_{0}]=\sum_{l=1}^{k}c_{i0}^{l}X_{l}\,.
\end{align}
By Eq.~\eqref{eq:poppmainintro}, $\popp = \sqrt{g}\, \nu^{1}\wedge \ldots\wedge \nu^{k}\wedge \nu^{0}$ where $g=1/\sum_{i,j=1}^{k} (c_{ij}^{0})^{2}$. Then the statement follows from the identity
\begin{equation}\label{eq:normaliz0}
\|J\|^{2} = \sum_{i,j=1}^k \metr{X_i}{J X_j}^2 = \sum_{i,j=1}^k d\omega(X_i,X_j)^2 = \sum_{i,j=1}^k \omega([X_i,X_j])^2 = \sum_{i,j=1}^k (c_{ij}^{0})^{2}\,.
\end{equation}
Observe that, in the last equality of Eq.~\eqref{eq:normaliz0}, we employed Cartan formula for the differential of a one-form, and the fact that $\omega(X_i) = 0$.
\end{proof}

Eq.~\eqref{eq:poppcontact} can be expressed in terms of the eigenvalues of $J$. See also \cite[Remark 30]{corank1}, where the authors exhibit this formula for the case $\mc{G} = (4,5)$.

\brem
Let $f\in C^{\infty}(M)$ be a smooth, non-vanishing function. Then $\omega$ and $\omega':=f\omega$ define the same contact distribution $\distr$. However $d\omega'\neq fd\omega$ and, in general, the associated Reeb vector field is different. On the other hand, as a consequence of the identity $d\omega'|_{\distr}=fd\omega|_{\distr}$, it follows that $J'=f J$. Therefore, it is convenient to choose a \virg{normalized} contact form, which is uniquely specified (up to a sign) by the condition $\|J_q\|^2 = 1$, $\forall q \in M$. Then, in terms of the Reeb vector field associated with the normalized contact form, $\popp = \nu^{1}\wedge \ldots\wedge \nu^{k}\wedge \nu^{0}$. 
\erem

%
\subsubsection{Carnot groups of step 2}
A Carnot group $\mathbb{G}$ of step $2$ is a left-invariant sub-Riemannian structure on a nilpotent, connected, simply connected Lie group whose Lie algebra $\mathfrak{g}$ admits a  stratification $\mathfrak{g}=V_{1} \oplus  V_{2}$  with $[V_{1},V_{1}]=V_{2}$ and $[V_{1},V_{2}]=[V_{2}, V_{2}]=\{0\}$. The sub-Riemannian structure is defined by left translation of the subspace $V_{1}$, where we choose an orthonormal basis $X_1,\dots,X_k$.
It is possible to choose a basis $Y_{k+1},\ldots,Y_{n}$ of $V_{2}$ such that
$$[X_{i},X_{j}]=\sum_{h=k+1}^{n}b_{ij}^{h}Y_{h},\qquad [X_{i},Y_{h}]=[Y_{h},Y_{l}]=0.$$
Using the standard exponential coordinates (i.e. the identification of the Lie group and its Lie algebra via the exponential map) the explicit expression for the associated left-invariant vector fields in $\R^{n}=\{(x,y)\,|\, x\in \R^{k},y\in \R^{n-k}\}$ is 
\begin{gather}
X_i=\partial_{x_i}+\frac12 \sum_{j,h} b_{ij}^{h} x_j \partial_{y_{h}}\,, \qquad i=1,\ldots,k\,, \\
Y_{h}=\partial_{y_h}\,,\qquad  h=k+1,\ldots,n\,.
\end{gather}
In \cite{corank2}, the authors employed the skew-symmetric matrices $L^h$, $k+1\leq h\leq n$, of components $[L^{h}]_{ij}=b_{ij}^{h}$ in order to investigate the nilpotent approximation of a step 2 sub-Riemannian structure. In terms of these matrices, 
\bqn
B^{hl} = (L^h, L^l)\,,
\eqn
where $(M,N):= \text{Tr}(M^T N)$ is the Hilbert-Schmidt inner product on $\text{GL}(k,\mathbb{R})$. If the $L$ matrices are orthonormal, Eq.~\eqref{eq:poppmainintro} gives
\bqn
\popp = dx^1\wedge \ldots \wedge dx^k \wedge dy^{k+1} \wedge \ldots \wedge dy^n\,.
\eqn
The last formula is (up to a constant factor) the definition of Popp's volume employed in \cite[Definition 4]{corank2} and \cite{corank2g}, given in terms of a global adapted frame.

\subsubsection{Non-equiregular case}
The basic example of a bracket-generating, non-equiregular sub-Riemannian structure is the so-called \emph{Martinet distribution}. This is the distribution on $\mathbb{R}^3$ defined by the kernel of the one-form $\theta := dz - y^2 dx$. \rosso{A global frame for $\distr$, which we declare orthonormal,}  is
\bqn
X = \partial_x + y^2 \partial_z\,,\qquad
Y = \partial_y\,.
\eqn
Let $Z:=\partial_z$. Then $[X,Y] = -2y Z$ and $[Y,[X,Y]] = 2 Z$. Observe that $X,Y,Z$ is a global (adapted) frame for $TM$, therefore Martinet distribution is bracket-generating. However, its growth vector is
\bqn
\mc{G}(x,y,z) = \begin{cases}
(2,3) & \text{if }y \neq 0\,, \\
(2,2,3) & \text{if }y = 0 \,,
\end{cases}
\eqn
and the distribution is not equiregular on the hyperplane $y=0$. Nevertheless, if we restrict to the connected components of $\{y\neq 0\}$, we obtain a step $2$ equiregular sub-Riemannian manifold. Here, Theorem~\ref{t:poppmain} gives the following expression:
\bqn\label{eq:poppmartinet}
\popp = \frac{1}{|y|}dx\wedge dy \wedge dz\,.
\eqn
Eq.~\eqref{eq:poppmartinet} shows that singularities arise precisely on the hypersurface where the equiregularity hypotesis fails. In \cite{boscain-laurent}, the authors investigate the properties of the sub-Laplacian associated with this volume in the Martinet structure. They show that the sub-Laplacian is essentially self-adjoint in each connected component of $\{y\neq0\}$, hence the hyperplane $\{y=0\}$ acts as a barrier for the heat propagation.

\section{Sub-Laplacian}
In this section we define the canonical sub-Laplacian associated with a generic volume form and we prove Corollary~\ref{c:sublapl}, namely an explicit formula for the sub-Laplacian associated with Popp's volume.

On a Riemannian manifold, the Laplace-Beltrami operator is defined as the divergence of the gradient. This definition can be easily generalized to the sub-Riemannian setting.

\bdeff
Let $f\in C^{\infty}(M)$. The \emph{horizontal gradient} of $f$ is the unique horizontal vector field $\grad f$ such that
\begin{equation}
\metr{\grad f}{X} = X(f)\,, \qquad \forall X \in \Gamma(\distr)\,.
\end{equation}
\edeff
It follows from the definition that, in terms of a local frame $X_1,\dots,X_k$ for $\distr$
\begin{equation}\label{eq:grad}
\grad f =\sum_{i=1}^k X_i(f) X_i\,.
\end{equation}

\bdeff
Let $\mu \in \Omega^n(M)$ be a positive volume form, and $X \in \VecM$. The \emph{$\mu$-divergence} of $X$ is the smooth function $\dive_\mu X$ defined by 
\begin{equation}
\mathcal{L}_X \mu = \dive_\mu X\mu \,.
\end{equation}
where $\mathcal{L}_X$ is the Lie derivative in the direction $X$.
\edeff
Notice that the definition of divergence does not depend on the orientation of $M$, namely the sign of $\mu$. The divergence measures the rate at which the volume of a region changes under the integral flow of a field. Indeed, for any compact $\Omega\subset M$ and $t$ sufficiently small, let $e^{tX} : \Omega \to M$ be the flow of $X \in \VecM$, then
\begin{equation}
\left.\frac{d}{dt}\right|_{t=0} \int_{e^{tX}(\Omega)} \mu = - \int_\Omega \dive_\mu X \mu \,.
\end{equation}
The next proposition is sometimes employed as an alternative definition of divergence. Let $C^\infty_0(M)$ be the space of smooth functions with compact support.
\begin{proposition}\label{t:divtheo}
For any $f\in C_0^\infty(M)$ and $X\in \VecM$
\bqn
\int_M f \dive_\mu X \mu = - \int_M X(f) \mu\,.
\eqn
\end{proposition}
\begin{proof}
The proof is an easy consequence of the definition of $\mu$-divergence.
\end{proof}

The next lemma gives the relation between divergences associated with different volumes.
\begin{lemma}\label{l:divtrans}
Let $\mu, \mu' \in \Omega^n(M)$ be volume forms. Let $f \in C^{\infty}(M)$, $f \neq 0$ such that $\mu' = f \mu$. Then, for any $ X\in \VecM$
\begin{equation}\label{eq:divtrans}
\dive_{\mu'} X  = \dive_\mu X + X(\log f) \,.
\end{equation}
\end{lemma}
\begin{proof}
It follows from the Leibniz rule $\mc{L}_X(f\mu)=(Xf)\mu+f\mc{L}_X\mu=(X(\log f)+\dive_\mu X) f\mu$.
\end{proof}
When no confusion may arise, we write \virg{$\dive$}, without any reference to the volume form $\mu$. In the following, we fix the reference volume to be Popp's one. Lemma~\ref{l:divtrans} can be used to generalize the results to the case of a generic $\mu$-divergence.

With a divergence and a gradient at our disposal, we are ready to define the sub-Laplacian associated with the volume form $\mu$.
\bdeff
Let $\mu \in \Omega^n(M)$, $f \in C^{\infty}(M)$. The \emph{sub-Laplacian} associated with $\mu$ is the second order differential operator
\begin{equation}
\Delta f := \dive\lp \grad f\rp\,,
\end{equation}
\edeff
This definition reduces to the Laplace-Beltrami operator when $\mu$ is the Riemannian volume. \rosso{As a consequence of Eq.~\eqref{eq:grad} and the Leibniz rule for the divergence $\dive(fX)=Xf+f\,\dive(X)$, we can find the expression of the sub-Laplacian in terms of any local frame $X_1,\dots,X_k$:}
\begin{equation}
\rosso{\dive \lp \grad f \rp = \sum_{i=1}^k \dive\lp X_i(f)X_i\rp =\sum_{i=1}^k X_i(X_i(f)) + \dive (X_i) X_i(f)\,}.
\end{equation}
Then
\begin{equation}\label{eq:sublaplframe}
\lapl = \sum_{i=1}^k X_i^2 + \dive (X_i) X_i\,.
\end{equation}
\begin{remark}
Observe that the second order term of $\lapl$, namely the \virg{sum of squares} in Eq.~\eqref{eq:sublaplframe}, does not depend on the choice of the volume. Indeed, only the first order terms depend on it through the divergence operator, which changes according to Lemma~\ref{l:divxi} upon a change of volume.
\end{remark}
\begin{remark}
If we apply Proposition~\ref{t:divtheo} to the horizontal gradient $\grad g$, we obtain
\bqn
\int_M f \lapl g \mu= -\int_M \metr{\grad f}{\grad g} \mu\,, \qquad \forall f,g \in C^\infty_0(M)\,.
\eqn
Then $\lapl$ is symmetric and negative on $C_0^\infty(M)$. It can be proved that it is also essentially self-adjoint (see \cite{strichartz}).
\end{remark}

Now we prove a useful formula for the divergence associated with Popp's volume. Analogous formulae for $\mu$-divergences are easily obtained by an application of Lemma~\ref{l:divtrans}.
\begin{lemma}\label{l:divxi}
Let $X_1,\dots,X_n$ be a local adapted frame. Let $\dive$ be the divergence associated with Popp's volume. Then, for $i =1,\dots,n$
\begin{equation}\label{eq:divxi}
\dive X_i = - \lp \frac{1}{2}\sum_{j=1}^m \tx{Tr}(B^{-1}_j X_i(B_j)) + \sum_{l=1}^n c_{il}^l\rp  X_i\,.
\end{equation}
\end{lemma}
\begin{proof}
Let $\nu \in \Omega^1(M)$, and $X,Y \in \VecM$. The Lie derivative obeys Leibniz rule:
\bqn
\mathcal{L}_X \left(\nu (Y)\right) = (\mathcal{L}_X \nu )(Y) +\nu(\mathcal{L}_X Y)\,.
\eqn
Then, if $\nu^1,\dots,\nu^n$ is the dual frame associated with $X_1,\dots,X_n$ 
\bqn\label{eq:sc0dual}
\mathcal{L}_{X_i} \nu^j  = - \sum_{l=1}^n c_{il}^j \nu^l\,,
\eqn
which is the \virg{dual formulation} of Eq.~\eqref{eq:sc0intro}. By Theorem~\ref{t:poppmain}, Popp's volume is
\bqn
\popp=\frac{1}{\sqrt{\prod_{j} \det B_{j}}} \,\nu ^1 \wedge \ldots \wedge \nu^n \,.
\eqn
Then, for $i=1,\dots,n$,
\begin{multline}\label{eq:multline}
\mathcal{L}_{X_i} \popp =  \sqrt{\prod_{j} \det B_{j}}\, X_i\lp \frac{1}{\sqrt{\prod_{j} \det B_{j}}} \rp  \popp +\\ + \frac{1}{\sqrt{\prod_{j} \det B_{j}}}\lp \mathcal{L}_{X_i}\nu^1\wedge \ldots \wedge \nu^n + \ldots + \nu^1\wedge\ldots\wedge \mathcal{L}_{X_i} \nu^n \rp \,.
\end{multline}
Eq.~\eqref{eq:divxi} now follows from the definition of divergence, Eq.~\eqref{eq:sc0dual} and Eq.~\eqref{eq:multline}.
\end{proof}

Finally, Corollary~\ref{c:sublapl} is a straightforward consequence of Lemma~\ref{l:divxi} and Eq.~\eqref{eq:sublaplframe}. \qed

\section{Volume preserving transformations}

This section is devoted to the proof of \rosso{Propositions}~\ref{p:volumepres} and~\ref{p:volumetrans}.

\subsection{Proof of \rosso{Proposition}~\ref{p:volumepres}}

Let $\phi \in \iso (M)$ be a (local) isometry, and $1\leq i \leq m$. The differential $\phi_*$ induces a linear map 
\bqn
\widetilde{\phi}_* : \otimes^i \distr_q \to \otimes^i \distr_{\phi(q)}\,.
\eqn
Moreover $\phi_*$ preserves the flag $\distr \subset \ldots \subset \distr^m$. Therefore, it induces  a linear map
\bqn
\widehat{\phi}_*: \distr^i_q/\distr^{i-1}_q \to \distr^i_{\phi(q)}/\distr^{i-1}_{\phi(q)}\,.
\eqn
The key to the proof of \rosso{Proposition}~\ref{p:volumepres} is the following lemma.

\begin{lemma}
$\widetilde{\phi}_*$ and $\widehat{\phi}_*$ are isometries of inner product spaces.
\end{lemma}
\begin{proof}
The proof for $\widetilde{\phi}_*$ is trivial. The proof for $\widehat{\phi}_*$ is as follows. Remember that the inner product on $\distr^i/\distr^{i-1}$ is induced by the surjective maps $\pi_i : \otimes^i \distr \to \distr^i/\distr^{i-1}$ defined by Eq.~\eqref{eq:pimap}. Namely, let $Y \in \distr^i_q/\distr^{i-1}_q$. Then
\bqn
\|Y\|_{\distr^i_q/\distr^{i-1}_q} = \min\{ \|Z\|_{\otimes \distr_q} \text{ s.t. } \pi_i(Z) = Y\}\,.
\eqn
As a consequence of the properties of the Lie brackets, $\pi_i \circ \widetilde{\phi}_* = \widehat{\phi}_* \circ \pi_i$. Therefore
\bqn
\|Y\|_{\distr^i_q/\distr^{i-1}_q} = \min\{ \|\widetilde{\phi}_* Z\|_{\otimes \distr_{\phi(q)}} \text{ s.t. } \pi_i(\widetilde{\phi}_* Z) = \widehat{\phi}_* Y\} = \| \widehat{\phi}_* Y \|_{\distr^i_{\phi(q)}/\distr^{i-1}_{\phi(q)}}\,.
\eqn
By polarization, $\widehat{\phi}_*$ is an isometry.
\end{proof}

Since $\tx{gr}_q(\distr) = \oplus_{i=1}^m \distr_q^i/\distr_q^{i-1}$ is an orthogonal direct sum, $\widehat{\phi}_*: \tx{gr}_q(\distr) \to \tx{gr}_{\phi(q)}(\distr)$ is also an isometry of inner product spaces. 

Finally, Popp's volume is the canonical volume of $\tx{gr}_q(\distr)$ when the latter is identified with $T_q M$ through any choice of a local adapted frame. Since $\phi_*$ is equal to $\widehat{\phi}_*$  under such an identification, and the latter is an isometry of inner product spaces, the result follows. \qed

\subsection{Proof of \rosso{Proposition}~\ref{p:volumetrans}}

Let $\mu$ be a volume form such that $\phi^* \mu = \mu$ for any isometry $\phi \in \iso (M)$. There exists $f\in C^\infty(M)$, $f\neq 0$ such that $\popp = f \mu$. It follows that, for any $\phi \in \iso (M)$
\bqn
f \mu = \popp = \phi^* \popp = (f\circ \phi)\, \phi^* \mu = (f\circ \phi)\, \mu \,,
\eqn
where we used the $\iso (M)$-invariance of Popp's volume. Then also $f$ is $\iso (M)$-invariant, namely $\phi^* f = f$ for any $\phi \in \iso (M)$. By hypothesis, the action of $\iso (M)$ is transitive, then $f$ is constant. \qed

\vspace{0.2cm}
{\bf Acknowledgements.} The first author has been supported by the European Research Council, ERC StG 2009 \virg{GeCoMethods}, contract number 239748, by the ANR Project GCM, program \virg{Blanche}, project number NT09-504490.

{\small
\bibliographystyle{siam}
\bibliography{bibpopp}
}

\end{document}